\documentclass[12pt]{article}
\usepackage{graphicx}
\usepackage{amsmath}
\usepackage{amsfonts}
\usepackage{amsthm}
\usepackage{color}
\usepackage[utf8]{inputenc}
\usepackage[T1]{fontenc}
\usepackage{url}

\usepackage{chicago}
\newcommand {\citeAY} [1] {\citeNP {#1}}%
\newcommand {\citeAPY}[1] {\citeN  {#1}}%
\newcommand\eq[1] {(\ref{#1})}

\newcommand{\bfm}[1]{\mbox{\boldmath ${#1}$}}
\newcommand{\nonum}{\nonumber \\}
\newcommand{\beqa}{\begin{eqnarray}}
\newcommand{\eeqa}[1]{\label{#1}\end{eqnarray}}
\newcommand{\beq}{\begin{equation}}
\newcommand{\eeq}[1]{\label{#1}\end{equation}}
\newcommand{\Grad}{\nabla}
\newcommand{\Div}{\nabla \cdot}
\newcommand{\Curl}{\nabla \times}

\newcommand{\Md}{\partial}

\newcommand{\Gd}{\delta}
\newcommand{\Ge}{\epsilon}

\newcommand{\Gl}{\lambda}

\newcommand{\Go}{\omega}
\newcommand{\Gx}{\xi}
\newcommand{\Gy}{\psi}

\newcommand{\GF}{\Phi}

\newcommand{\BGm}{\bfm\mu}

\newcommand{\BGx}{\bfm\xi}

\newcommand{\BGF}{\bfm\Phi}

\def\Ba{{\bf a}}
\def\Bb{{\bf b}}

\def\Bk{{\bf k}}

\def\Bm{{\bf m}}
\def\Bn{{\bf n}}

\def\Bq{{\bf q}}

\def\Bu{{\bf u}}
\def\Bv{{\bf v}}

\def\Bx{{\bf x}}

\def\BC{{\bf C}}

\def\BG{{\bf G}}

\def\BI{{\bf I}}

\def\BL{{\bf L}}

\def\BX{{\bf X}}

\def \ba {\begin{array}}
\def \ea {\end{array}}

\def \refe #1.{(\ref{#1})}
\def \reff #1.{figure~\ref{#1}}
\def \refs #1.{section~\ref{#1}}
\def \refss #1.{subsection~\ref{#1}}
\def \refD #1.{Definition~\ref{#1}}
\def \refT #1.{Theorem~\ref{#1}}
\def \refL #1.{Lemma~\ref{#1}}
\def \refC #1.{Corollary~\ref{#1}}
\def \refR #1.{Remark~\ref{#1}}
\def \refE #1.{Example~\ref{#1}}
\def \refN #1.{Notation~\ref{#1}}

\usepackage{graphics}
\usepackage{amssymb}
\newtheorem{Theorem}{Theorem}[section]
\newtheorem{Lemma}[Theorem]{Lemma}

\newtheorem{Remark}[Theorem]{Remark}
\newtheorem{Definition}[Theorem]{Definition}
\title{High frequency homogenization for travelling waves in periodic media}
\author{D. Harutyunyan\footnote{Department of Mathematics,
University of Utah, Salt Lake City, Utah 84112, USA }, G. W. Milton\footnote{Department of Mathematics, University of Utah, Salt Lake City, Utah 84112, USA } and R. V. Craster\footnote{Department of Mathematics, Imperial College London, SW7 2AZ, UK}}

\date{\today}

\begin{document}
\maketitle

\begin{abstract}

We consider high frequency homogenization in periodic media for
travelling waves of several different equations: the wave equation for scalar-valued waves such as acoustics;
the wave equation for vector-valued waves such as electromagnetism and elasticity; and a system that encompasses the Schr{\"o}dinger equation. This homogenization
applies when the wavelength is of the order of the size of the medium periodicity cell. The travelling wave is assumed to be the sum of two waves: a modulated Bloch carrier
wave having crystal wave vector $\Bk$ and frequency $\omega_1$ plus a modulated Bloch carrier wave having crystal wave vector $\Bm$ and frequency $\omega_2$.
We derive effective equations for the modulating functions, and then
prove that there is no coupling in the effective equations between the
two different waves both in the scalar and the system cases. To be
precise, we prove that there is no coupling unless $\omega_1=\omega_2$
and $(\Bk-\Bm)\odot\Lambda \in 2\pi \mathbb Z^d,$ where
$\Lambda=(\lambda_1\lambda_2\dots\lambda_d)$ is the periodicity cell of the medium and for any two vectors
$a=(a_1,a_2,\dots,a_d), b=(b_1,b_2,\dots,b_d)\in\mathbb R^d,$ the
product $a\odot b$ is defined to be the vector
$(a_1b_1,a_2b_2,\dots,a_db_d).$ This last condition
forces the carrier waves to be equivalent Bloch waves meaning that the coupling constants in the system of effective equations vanish.
 We use two-scale analysis and some new weak-convergence type lemmas. The analysis is not at the same level of rigor as that of Allaire and coworkers who use two-scale convergence theory to treat the problem, but has the advantage of simplicity which will allow it to be easily extended to the case where there is degeneracy of the Bloch
eigenvalue.
\end{abstract}
%%%% Subject entries to be placed here %%%%
\textbf{Wave motion, Applied Mathematics}\\
\vspace{0.2cm}

%%%% Keyword entries to be placed here %%%%
\textbf{Asymptotics, Bloch waves, Homogenization}

%%%% Insert corresponding author and its email address}

\section{Introduction}
\label{Bsec:1}
\setcounter{equation}{0}
%%%%%%%%%%%%%%%%%%%%%%%%%%%%%%%%%%%%%%%%%%%%%%%%%%%%%%%%%%%%%%%%%%%%%%%%%%%

Periodic materials, or at least almost periodic materials abound in
nature: crystals are one of the most obvious, and prevalent,
examples. Opals are another example,
which consist of tiny spherical particles of silica arranged in a face-centered cubic array, which act like a diffraction grating to create the beautiful colors we see
 (\citeAY{Sanders:1964:CPO}; \citeAY{Greer:1969:SSA}). A sea mouse has a wonderful iridescence which is caused by a hexagonal array of voids in a matrix of chitin
(\citeAY{Parker:2001:AEA}). Recently it has been discovered that chameleons change their color by adjusting the lattice spacing of guanine nanocrystals in their
skins (\citeAY{Teyssier:2015:PCC}).
The word honeycomb is associated with bees, and the giant's causeway in Ireland consists of a hexagonal array of Basalt columns. Many patterns of tiles are periodic. Beautiful periodic structures, now also can be tailor made
using three dimensional lithography and printing techniques  (\citeAY{Pendry:2004:RLN}; \citeAY{Kadic:2012:PPM}; \citeAY{Buckmann:2012:TMM}; \citeAY{Buckmann:2014:TDD}; \citeAY{Meza:2014:SLR})
and of course two-dimensional periodic structures are even easier to produce (\citeAY{Bragg:1947:DMC}; \citeAY{Krauss:1996:TDP}; \citeAY{Yu:2014:FOD}).

There has of course been tremendous interest in the properties of periodic structures. The electronic properties of periodic structures were extensively studied (see, for example, \citeAY{Kittel:2005:ISS}) it being realized that the
band structure of the dispersion diagram for Schr\"{o}dinger's equation is intimately connected with whether a material is a conductor, insulator, or semiconductor, and type of semiconductor (according to whether there was a direct
gap or indirect gap). Then came the realization that the same concepts of dispersion diagrams and band gaps also apply at a macroscopic scale, to electromagnetic, and elastic wave propagation through periodic composite materials
(\citeAY{Bykov:1975:SEM}; \citeAY{John:1987:SLP}; \citeAY{Yablonovitch:1987:ISE}; \citeAY{Sigalas:1993:BSE}; \citeAY{Movchan:2007:BFB}). This lead to explosive growth in the area. An immense bibliography on the subject, with over 12,000 papers, approximately doubling every two-years since 1987 (until 2008, which was when the bibliography ceased being updated) was complied by Dowling (see http://www.phys.lsu.edu/\textasciitilde jdowling/pbgbib.html). For an excellent review of the subject see the book by \citeAPY{Joannopoulos:2008:ISS} [see also the beautiful article by \citeAPY{Gorishnyy:2005:SI} on acoustic band gap materials].

By suitably adapting the high-frequency homogenization approach of  \citeAPY{Craster:2010:HFH} we prove, that for different travelling waves in periodic medium, the effective equations in the bulk of the material for the function that modulates the wave do not {\it couple}, i.e., waves having wavevectors $\Bk/\Ge$ do not couple with waves having wavevectors $\Bm/\Ge$: here $\Ge$ is a small parameter characterizing the length of the unit cell of periodicity, and we are looking at the homogenization limit $\Ge\to 0$. Thus the scaling is such that the short scale oscillations in the waves are on the same scale as the length of the unit cell, in contrast with the usual low frequency homogenization where the wavelength is much larger than the size of the unit cell.
We assume, for simplicity, in this first analysis that the Bloch equations are non-degenerate at these wave-vectors.
A treatment without this assumption has been given by Brassart and Lenczner \citeAY{Brassart:2010:TSM} for the wave equation.
Then to leading order we find that
for the waves $(\Bk,\omega_1)$ and $(\Bm,\omega_2)$, the field (or potential) that solves the equations takes the form
\beq u(t,\Bx)\approx f_0^{(1)}(t,\Bx)V_0^{(1)}(\Bx/\Ge)e^{-i(\Bk\cdot(\Bx/\Ge)-\omega_1 (t/\Ge))}+ \\
    f_0^{(2)}(t,\Bx)V_0^{(2)}(\Bx/\Ge)e^{-i(\Bm\cdot(\Bx/\Ge)-\omega_2 (t/\Ge))}.
\eeq{B.1}
Here $V_0^{(1)}(\Bx/\Ge)e^{-i(\Bk\cdot(\Bx/\Ge)-\omega_1 (t/\Ge))}$ and $V_0^{(2)}(\Bx/\Ge)e^{-i(\Bm\cdot(\Bx/\Ge)-\omega_2 (t/\Ge))}$ are the Bloch solutions at the wavevectors $\Bk/\Ge$ and $\Bm/\Ge$,
( in which $V_0^{(1)}(\Bx/\Ge)$ and $V_0^{(2)}(\Bx/\Ge)$ have the same unit cell of periodicity as the periodic material we are considering) and the modulating functions $f_0^{(1)}(t,\Bx)$ and  $f_0^{(2)}(t,\Bx)$
satisfy the {\it homogenized equation}
\begin{equation}
\label{B2}
\nabla g\cdot\nabla f_0^{(\ell)}(t,\Bx)+\frac{\partial f_0^{(\ell)}(t,\Bx)}{\partial t}=0,\quad \ell=1,2,
\end{equation}
and $\Go/\Ge=g(\Bk/\Ge)$ is the dispersion relation. This effective equation admits as solutions, the expected travelling waves
\beq f_0^{(1)}(t,\Bx) =  h_1(\Bv_1\cdot\Bx-t),\quad f_0^{(2)}(t,\Bx)=h_2(\Bv_2\cdot\Bx-t)
\eeq %{B3}
where $h_1$ and $h_2$ are arbitrary functions and $\Bv_1$ and $\Bv_2$ are the group velocities which satisfy
$\Bv_i\cdot\nabla g=1$. In the time harmonic case, where the waves are not travelling the first results on high frequency homogenization are those
of \citeAPY{Birman:2006:HMP}, which provides a rigorous justification of high frequency homogenization. That paper is difficult to follow unless one is an expert in spectral theory, so in the appendix we make the connection between the paper of \citeAPY{Birman:2006:HMP} and that of \citeAPY{Craster:2010:HFH}. The approach of  \citeAPY{Craster:2010:HFH} is straightforward and very reminiscent of the standard formal approach to homogenization (see, for example, \citeAY{Bensoussan:1978:AAP}, who furthermore homogenize a Schrodinger equation at high frequency).
One treats the large and small length scales as independent variables $\BX$ and $\xi$, that are coupled when one replaces in the governing equations any derivative $\Md/\Md x_j$ with $\Md/\Md X_j+(1/\Ge)\Md/\Md \Gx_j$. This approach is not
rigorous, so will need to be  supplemented at some stage, by either a rigorous analysis, or by supporting numerical calculations. \citeAPY{Hoefer:2011:DMH}
have done a careful rigorous analysis, with error estimates, for high frequency homogenization applied to the Schr\"odinger equation,
where they keep higher terms in the expansion.

It is also to be emphasised that there has been extensive work by Allaire and
co-workers in this area using the ideas of two-scale convergence introduced by \citeAPY{Nguetseng:1989:GCR} and \citeAPY{Allaire:1992:HTS}
particularly for the acoustic equation
\begin{equation}
\label{B3}
\nabla\cdot\left(\mathbf{a}\left(\frac{\mathbf{x}}{\epsilon}\right)\nabla u\right)=b\left(\frac{\mathbf{x}}{\epsilon}\right)\frac{\partial^2 u}{\partial t^2},\qquad \mathbf{x}=(x_1,x_2,\dots,x_d)\in\mathbb R^d,
\end{equation}
assuming ellipticity for the tensor $a$, and positivity for the scalar
$b$ (\citeAY{Allaire:2009:DBW}; \citeAY{Allaire:2011:DGO}; \citeAY{Brassart:2010:TSM}).
These assumptions are also needed in our analysis to ensure unique solvability of the Bloch equation. The work of \citeAPY{Allaire:2011:DGO} goes further than us
in that they prove that the enveloping function converges to that predicted by the two-scale analysis, and that they prove that the enveloping function obeys a Schr{\"o}dinger type equation as expected from the paraxial approximation. At the level of our analysis the dispersion of the enveloping function is absent, so
further work needs be done to account for it. We also emphasize that there is much older work in the book of \citeAY{Bensoussan:1978:AAP} on high frequency homogenization of the Schr\"odinger equation. Particularly, we draw the reader's attention to the effective equations (4.33) and the formulae for the effective moduli both expressed in terms of the derivatives of the dispersion formula and in terms of the solution to a cell problem (see the discussion at the bottom of page 352).
In our analysis we treat electromagnetism and elasticity in one
single stroke in Section 5, and we treat a general class of equations which includes the Scr{\"o}dinger equation in Section 6. Again we note
that \citeAPY{Allaire:2005:HSE} treated the Schr\"{o}dinger equation and \citeAPY{Allaire:2013:DBW} treated Maxwell's equations using the tool of two-scale
convergence.

Although much of the work discussed thus far is analytical, it is useful to note that
%Other applications of high frequency homogenization are to shear,
%antiplane waves (\citeAY{Craster:2010:HFH}), to in-plane elasticity,
%(\citeAY{Antonakakis:2014:HEP}), to the plate equation
%(\citeAY{Antonakakis:2012:HFA}), to media pinned at lattice points
%(\citeAY{Makwana:2015:WMM}), and to elecromagnetism
%(\citeAY{Antonakakis:2013:AMP}; \citeAY{Maling:2015:HME}). Not only
%has the theory been developed in these cases, but also
these types of effective media have been tested numerically and
applied to interpret and design experiments \cite{Ceresoli:2015:DEA}. The most striking behavior occurs when the effective equations which describe the
macroscopic modulation of the waves are hyperbolic rather than elliptic: then the radiation concentrates along the characteristic lines and the star shaped patterns
predicted by high frequency homogenization are seen in full finite element simulations (see the figures in \citeAY{Makwana:2015:WMM} and \citeAY{Antonakakis:2014:HEP}).

The main factor which influences the macroscopic equations one gets is
the degeneracy of the wavefunctions associated with the expansion point (see, for example, \citeAY{Antonakakis:2013:HFH}). The simplest case,
and the one first treated by  \citeAPY{Birman:2006:HMP}, \citeAPY{Craster:2010:HFH} is when there is no degeneracy.
For simplicity, and because this is the generic case, we will also assume there is no degeneracy. The next section is then the homogenization
of a model equation.

\setcounter{equation}{0}
%%%%%%%%%%%%%%%%%%%%%%%%%%%%%%%%%%%%%%%%%%%%%%%%%%%%%%%%%%%%%%%%%%%%%%%%%%%%%%%%%%%%%%%%%%%%%%55555
%\newtheorem{Theorem}{Theorem}[section]
%\newtheorem{Lemma}[Theorem]{Lemma}
%\newtheorem{Corollary}[Theorem]{Corollary}
%\newtheorem{Remark}[Theorem]{Remark}
%\newtheorem{Definition}[Theorem]{Definition}
\section{A scalar valued wave traveling in a periodic medium}
\label{Bsec:2}
Our first aim is to homogenize the model equation
\begin{equation}
\label{B2.1}
\nabla\cdot\left(\mathbf{a}\left(\frac{\mathbf{x}}{\epsilon}\right)\nabla u\right)=b\left(\frac{\mathbf{x}}{\epsilon}\right)\frac{\partial^2 u}{\partial t^2},\qquad \mathbf{x}=(x_1,x_2,\dots,x_d)\in\mathbb R^d,
\end{equation}
where $\mathbf{a}\colon\mathbb R^d\to\mathbb R^{d\times d}$ is a symmetric matrix and $b\colon\mathbb R^d\to\mathbb R$ are also cell periodic with the same cell of periodicity. We assume, for simplicity, that the unit cell is a rectangular prism, though of course we expect the analysis to go through for any
Bravais lattice. This is the equation of acoustics when $u$
is the pressure, $b({\mathbf{x}}/{\epsilon})$ is inverse of the bulk modulus of the fluid, and $\mathbf{a}({\mathbf{x}}/{\epsilon})$ is the inverse
of the density,  which we allow to be anisotropic (as may be the case in metamaterials).

We rewrite the equation in the form

\begin{equation}
\label{B2.2}
\begin{pmatrix}
\frac{\partial}{\partial t}\\
\nabla
\end{pmatrix}\cdot \mathbf{C}
\begin{pmatrix}
\frac{\partial}{\partial t}\\
\nabla
\end{pmatrix}u(\mathbf{x},t)=0,\qquad
\text{where}\qquad
\mathbf{C}=
\begin{pmatrix}
-b & 0\\
0 & \mathbf{a}
\end{pmatrix}.
\end{equation}
Denoting the new variable $x_0=t$ and
$$\overline{\nabla}=
\begin{pmatrix}
\frac{\partial}{\partial x_0}\\
\nabla
\end{pmatrix},
$$
we get the equation
\begin{equation}
\label{B2.3}
\overline{\nabla}\cdot(\mathbf{C}\overline{\nabla}u(\mathbf{x},t))=0.
\end{equation}
As is standard in homogenization theory we replace $\overline{\nabla}$ by $\overline{\nabla}_{X}+\frac{1}{\epsilon}\overline{\nabla}_{\xi},$
where $\BX=(x_0,x_1,x_2,\dots,x_d)$ is the slow variable and
$\BGx=(\xi_0,\xi_1,\xi_2,\dots,\xi_d)=\frac{X}{\epsilon}$ is the fast variable. The motivation for this is that if
we have a function $g(\Bx,\Bx/\Ge)=g(\BX,\BGx)$, then $\overline{\nabla}$ acting on $g(\Bx,\Bx/\Ge)$ gives the same result
as $\overline{\nabla}_{X}+\frac{1}{\epsilon}\overline{\nabla}_{\xi}$ acting on $g(\BX,\BGx)$, with $\BGx$ and $\BX$ treated
as independent variables. Thus we are scaling space and time in the same way as we believe is appropriate when
the dispersion diagram is such that $\frac{\partial\omega}{\partial\Bk}$ has a nonzero finite value. At points where
$\frac{\partial\omega}{\partial\Bk}$ is zero we do not believe this is an appropriate scaling as indicated by \citeAPY{Birman:2006:HMP}
and later by \citeAPY{Craster:2010:HFH}. Thus we assume that we are in the case when $\frac{\partial\omega}{\partial\Bk}$ has a nonzero finite value
and therefore making this replacement we arrive at
\begin{align}
\label{B2.4}
\overline{\nabla}_{\xi}\cdot(\mathbf{C}(\BGx)\overline{\nabla}_{\xi}u(\BX,\BGx))
+\epsilon\overline{\nabla}_{\xi}\cdot(\mathbf{C}(\BGx)\overline{\nabla}_{X}u(\BX,\BGx))\\ \nonumber
+\epsilon\overline{\nabla}_{X}\cdot(\mathbf{C}(\BGx)\overline{\nabla}_{\xi}u(\BX,\BGx))
+\epsilon^2\overline{\nabla}_{X}\cdot(\mathbf{C}(\BGx)\overline{\nabla}_{X}u(\BX,\BGx))=0.
\end{align}
Now we choose to homogenize waves which on the short length scale look like Bloch solutions, with frequencies
$\Go_1$ and $\Go_2$ and and wavevectors $\Bk$ and $\Bm$, but which are modulated on the long length scale. Our aim is to find the macroscopic
equation satisfied by the modulation. We assume that the wavenumber-frequency pairs $(\Bk,\omega_1)$ and
 $(\Bm,\omega_2)$ belong to the dispersion diagram. We will prove later in Section~\ref{Bsec:4}, that any two different waves do not interact (do not couple).

%\begin{figure}[htbp]
%\centering
%\includegraphics[width=0.6\textwidth]{dispdiag.pdf}
%\caption{The first two branches in a dispersion diagram are sketched in (a) for the simple case of a one-dimensional medium. The edges of the Brillouin zone %are shown by the dashed lines. In (b) we have taken a unit cell which is twice the size of the original, and now at the edges of each Brilloin zone, there is a crossing of branches,
%that correspond to wavevectors  $\Bk$ and $-\Bk$ of the orignal dispersion curve. }
%\label{dispdiag}
%\end{figure}

We have, that $u(\BX,\BGx)$ is the sum of two waves,
\begin{equation}
\label{B2.5}
u(\BX,\BGx)=u^{(1)}(\BX,\BGx)+u^{(2)}(\BX,\BGx),
\end{equation}
where for fixed $\BX$ the functions $u^{(1)}(\BX,\BGx)$, $u^{(2)}(\BX,\BGx)$ as functions of $\BGx$ are Bloch functions,
oscillating at frequencies $\Go_1$ and $\omega_2$ and having wavevectors $\Bk$ and $\Bm$ respectively.
Thus, the functions $e^{-i(\Bk\cdot\BGx'-\omega_1\xi_0)}u^{(1)}(\BX,\BGx)$ and $e^{-i(\Bm\cdot\BGx'-\omega_2\xi_0)}u^{(1)}(\BX,\BGx)$ are periodic in
$\BGx'=(\xi_1,\xi_2,\dots,\xi_d)$ and independent of $\xi_0.$ We seek a solution of equation (\ref{B2.4}) in the form
\begin{equation}
\label{B2.6}
u^{(i)}(\BX,\BGx)=u_0^{(i)}(\BX,\BGx)+\epsilon u_1^{(i)}(\BX,\BGx)+\epsilon^2u_2^{(i)}(\BX,\BGx)+\cdots,\qquad i=1,2.
\end{equation}
Plugging in the expressions of $u^{(1)}(\BX,\BGx)$ and $u^{(2)}(\BX,\BGx)$ in (\ref{B2.4}) we get at the zeroth order,
\begin{equation}
\label{B2.7}
\sum_{i=1}^2\overline{\nabla}_{\xi}\cdot(\mathbf{C}(\BGx)\overline{\nabla}_{\xi}u_0^{(i)}(\BX,\BGx))=0.
\end{equation}
From the uniqueness of the solutions to the Bloch equations, up to a multiplicative complex constant,
equation (\ref{B2.7}) implies that $u_0^{(1)}(\BX,\BGx)$ and $u_0^{(2)}(\BX,\BGx)$ can be separated in the fast and slow variables, i.e.,
\begin{equation}
\label{B2.8}
u_0^{(1)}(\BX,\BGx)=f_0^{(1)}(\BX)U_0^{(1)}(\BGx)\quad\text{and}\quad u_0^{(2)}(\BX,\BGx)=f_0^{(2)}(\BX)U_0^{(2)}(\BGx),
\end{equation}
where $U_0^{(1)}(\BGx)$ and $U_0^{(2)}(\BGx)$ solve the Bloch equations
\begin{equation}
\label{B2.9}
\overline{\nabla}_{\xi}\cdot(\mathbf{C(\BGx')}\overline{\nabla}_{\xi}U_0^{(i)}(\BGx))=0,\qquad i=1,2.
\end{equation}
and $f_0^{(1)}(\BX)$ and $f_0^{(2)}(\BX)$ are the modulating functions whose governing equation we seek to find.

It is then clear that $U_0^{(1)}(\BGx)$ and $U_0^{(2)}(\BGx)$ have the form
\begin{equation}
\label{B2.10}
U_0^{(1)}(\BGx)=V_0^{(1)}(\BGx')e^{-i(\Bk\cdot\BGx'-\omega_1\xi_0)}\quad\text{and}\quad U_0^{(2)}(\BGx)=V_0^{(2)}(\BGx')e^{-i(\Bm\cdot\BGx'-\omega_2\xi_0)},
\end{equation}
where $V_0^{(1)}$ and $V_0^{(2)}$ are cell-periodic functions of $\BGx'$, solving
\beqa
&~&\omega_1^2b(\BGx')V_0^{(1)}(\BGx')+(-i\Bk+\overline{\nabla}_{\xi'})\cdot \mathbf{a}(\BGx')(-i\Bk+\overline{\nabla}_{\xi'})V_0^{(1)}(\BGx')=0\\ \nonumber
&~&\omega_2^2b(\BGx')V_0^{(2)}(\BGx')+(-i\Bm+\overline{\nabla}_{\xi'})\cdot \mathbf{a}(\BGx')(-i\Bm+\overline{\nabla}_{\xi'})V_0^{(2)}(\BGx')=0.
\eeqa{B2.11}

At the first order we get the following equation

\begin{equation}
\label{B2.12}
\sum_{i=1}^2\overline{\nabla}_{\xi}\cdot(\mathbf{C}(\BGx')\overline{\nabla}_{\xi}u_1^{(i)}(\BX,\BGx))=
-\sum_{i=1}^2\left[\overline{\nabla}_{\xi}\cdot(\mathbf{C}(\BGx')\overline{\nabla}_{X}u_0^{(i)}(\BX,\BGx))+
\overline{\nabla}_{X}\cdot(\mathbf{C}(\BGx')\overline{\nabla}_{\xi}u_0^{(i)}(\BX,\BGx))\right]
\end{equation}

Next we take the complex conjugate of the equations in (\ref{B2.9}) to get
\begin{equation}
\label{B2.13}
\overline{\nabla}_{\xi}\cdot(\mathbf{C(\xi')}\overline{\nabla}_{\xi}U_0^{(i)\ast}(\BGx))=0,\qquad i=1,2,
\end{equation}
where $z^\ast$ denotes the complex conjugate of $z.$
Note, that equation (\ref{B2.12}) can be written in the following way
\begin{align}
\label{B2.14}
\sum_{l=1}^2\overline{\nabla}_{\xi}\cdot(\mathbf{C}(\BGx')&\overline{\nabla}_{\xi}u_1^{(l)}(\BX,\BGx))\\ \nonumber
&=-\sum_{l=1}^2\sum_{i,j=0}^d\frac{\partial f_0^{(l)}(\BX)}{\partial X_j}\left(2C_{ij}(\BGx')\frac{\partial U_0^{(l)}(\BGx)}{\partial \xi_i}
+\frac{\partial C_{ij}(\BGx')}{\partial \xi_i}U_0^{(l)}(\BGx)\right).
\end{align}
Assume now $Q$ is a large rectangular cell in the coordinate system $\BGx.$
Following the idea in (\citeAY{Bensoussan:1978:AAP}, page 307, see also \citeAPY{Craster:2010:HFH}, we multiply equation (\ref{B2.14}) by
$U_0^{(p)\ast}$ for $p=1,2$ and taking the average over $Q$ of both sides of the obtained identity we get
\begin{align}
\label{B2.15}
&\frac{1}{Q}\int_{|Q|}\sum_{l=1}^2U_0^{(p)\ast}\overline{\nabla}_{\xi}\cdot(\mathbf{C}(\BGx')
\overline{\nabla}_{\xi}u_1^{(l)}(\BX,\BGx)) d\BGx\\ \nonumber
&=
\sum_{l=1}^2\sum_{i,j=0}^d\frac{\partial f_0^{(l)}(\BX)}{\partial X_j}\frac{1}{|Q|}\int_{Q}U_0^{(p)\ast}\left(2C_{ij}(\BGx')\frac{\partial U_0^{(l)}(\BGx)}{\partial \xi_i}
+\frac{\partial C_{ij}(\BGx')}{\partial \xi_i}U_0^{(l)}(\xi)\right) d\BGx,
\end{align}
where $Q=\prod_{i=0}^d[0,a_i]$ and $|Q|$ is the volume of $Q.$

Let us show, that the left hand side of (\ref{B2.15}) goes to zero when $Q\to\infty$, by which we mean $a_i\to\infty$ for $i=0,1,2,\dots,d.$
Indeed, after an integration by parts, we get
\begin{align}
\label{B2.16}
&\int_{Q}\sum_{l=1}^2U_0^{(p)\ast}\overline{\nabla}_{\xi}\cdot(\mathbf{C}(\BGx')
\overline{\nabla}_{\xi}u_1^{(l)}(\BX,\BGx)) d\BGx=\\ \nonumber
&=\int_{\partial Q}\sum_{l=1}^2U_0^{(p)\ast}\Bn\cdot(\mathbf{C}(\BGx')
\overline{\nabla}_{\xi}u_1^{(l)}(\BX,\BGx))dS-\int_{Q}\sum_{l=1}^2\overline{\nabla}_{\xi}U_0^{(p)\ast}\mathbf{C}(\BGx')
\overline{\nabla}_{\xi}u_1^{(l)}(\BX,\BGx)d\BGx,
\end{align}
where $\Bn$ is the outward unit normal to $\partial Q.$
On the other hand we have by multiplying (\ref{B2.13}) by $u_1^{(l)}(\BX,\BGx)$
and integrating the obtained equality over $Q$ by parts we get
\begin{equation}
\label{B2.17}
\int_{\partial Q}\sum_{l=1}^2u_1^{(l)}(\BX,\BGx)\Bn\cdot(\mathbf{C}(\BGx')\overline{\nabla}_{\xi}U_0^{(p)\ast}(\BGx))dS
-\int_{Q}\sum_{l=1}^2\overline{\nabla}_{\xi}U_0^{(p)\ast}\mathbf{C}(\BGx')\overline{\nabla}_{\xi}u_1^{(l)}(\BX,\BGx)d\BGx=0.
\end{equation}
%Next, from the symmetry of the tensor $\BC$ we have that
%$$\int_{\overline{Q}}\sum_{l=1}^2\overline{\nabla}_{\xi}U_0^{(m)\ast}\mathbf{C}(\BGx')\overline{\nabla}_{\xi}u_1^{(l)}(\BX,\BGx)d\BGx
%=\int_{\overline{Q}}\sum_{l=1}^2\overline{\nabla}_{\xi}U_0^{(m)\ast}\mathbf{C}(\BGx')\overline{\nabla}_{\xi}u_1^{(l)}(\BX,\BGx)d\BGx,$$
Thus by combining (\ref{B2.16}) and (\ref{B2.17}) we get

\begin{align}
\label{B2.18}
&\int_{Q}\sum_{l=1}^2U_0^{(p)\ast}\overline{\nabla}_{\xi}\cdot(\mathbf{C}(\BGx')
\overline{\nabla}_{\xi}u_1^{(l)}(\BX,\BGx))d\BGx=\\ \nonumber
&=\int_{\partial Q}\sum_{l=1}^2u_1^{(l)}(\BX,\BGx)\Bn\cdot(\mathbf{C}(\BGx')\overline{\nabla}_{\xi}U_0^{(p)\ast}(\BGx))dS+
\int_{\partial Q}\sum_{l=1}^2U_0^{(p)\ast}\Bn\cdot(\mathbf{C}(\BGx')\overline{\nabla}_{\xi}u_1^{(l)}(\BX,\BGx))dS
\end{align}
Taking into account the continuity and the periodic structure of the functions $\overline{\nabla}_{\xi}u_1^{(l)}, \overline{\nabla}_{\xi}U_0^{(p)\ast}$
and the tensor $\BC(\BGx')$ we have the estimate
$$\left|\int_{Q}\sum_{l=1}^2U_0^{(p)\ast}\overline{\nabla}_{\xi}\cdot(\mathbf{C}(\BGx')
\overline{\nabla}_{\xi}u_1^{(l)}(\BX,\BGx))d\BGx\right|\leq M\mathcal{H}^{d-1}(\partial Q),$$
where $\mathcal{H}^{d-1}(\partial Q)$ is the $d-1$ dimensional Hausdorff measure in $\mathbb R^d,$ i.e., it is the surface measure, and $M$
is a sufficiently large constant. Our claim follows now from the obvious equality
$$\lim_{Q\to\infty}\frac{\mathcal{H}^{d-1}(\partial Q)}{|Q|}=0.$$
In other words this condition over the supercell $Q,$ which is the analog of the the solvability condition that was
over a unit cell in \citeAPY{Craster:2010:HFH}, gives us the following equations:
\begin{equation}
\label{B2.19}
\sum_{l=1}^2\sum_{j=0}^dd_{jp}^{(l)}\frac{\partial f_0^{(l)}(\BX)}{\partial X_j}=0, \quad \text{for} \quad  p=1,2,
\end{equation}
where the coefficients entering these homogenized equations are given by
\begin{align}
\label{B2.20}
&d_{jp}^{(l)}=\lim_{Q\to\infty}\sum_{i=0}^d\frac{1}{|Q|}\int_{Q}U_0^{(p)\ast}\left(2C_{ij}(\BGx')
\frac{\partial U_0^{(l)}(\BGx)}{\partial \xi_i}+\frac{\partial
  C_{ij}(\BGx')}{\partial \xi_i}U_0^{(l)}(\BGx)\right) d\BGx,\\ \nonumber
&j=0,1,\dots,d,\quad p,l=1,2.
\end{align}
As will be seen in the next Section~\ref{Bsec:4}, the formula (\ref{B2.20}) can be significantly simplified.

%%%%%%%%%%%%%%%%%%%%%%%%%%%%%%%%%%%%%%%%%%%%%%%%%%%%%%%%%%%%%%%%%%%%%%%%%%%%%%%%%%%%%%%
\setcounter{equation}{0}

%%%%%%%%%%%%%%%%%%%%%%%%%%%%%%%%%%%%%%%%%%%%%%%%%%%%%%%%%%%%%%%%%%%%%%%%%%%%%%%%%%%%%%%%%%%%55
\section{Wave coupling analysis}
\label{Bsec:4}
%%%%%%%%%%%%%%%%%%%%%%%%%%%%%%%%%%%%%%%%%%%%%%%%%%%%%%%%%%%%%%%%%%%%%%%%%%%%%%%%%%%%%%%%%%%%%
\setcounter{equation}{0}

In this section we consider the following question: is there interaction between any two different waves?
Let us look to see if there is interaction in the homogenized equations between waves
corresponding to points $(\Bk,\Go_1)$, and $(\Bm,\Go_2)$ on the dispersion diagram.
To that end, we must analyze the coupling coefficients in the homogenized equations. We have seen in Section~\ref{Bsec:2}, that the homogenized
equations are given by
\begin{equation}
\label{B4.1}
\sum_{l=1}^2\sum_{j=0}^dd_{jp}^{(l)}\frac{\partial f_0^{(l)}(\BX)}{\partial X_j}=0, \quad \text{for} \quad  p=1,2,
\end{equation}
where the coefficients entering these homogenized equations are given by
\begin{align}
\label{B4.2}
&d_{jp}^{(l)}=\lim_{Q\to\infty}\sum_{i=0}^d\frac{1}{|Q|}\int_{Q}U_0^{(p)\ast}\left(2C_{ij}(\BGx')
\frac{\partial U_0^{(l)}(\BGx)}{\partial \xi_i}+\frac{\partial C_{ij}(\BGx')}{\partial \xi_i}U_0^{(l)}(\BGx)\right)d\BGx,\\ \nonumber
&j=0,1,\dots,d,\quad p,l=1,2.
\end{align}
Denote furthermore by $\Lambda=\prod_{i=1}^d [0,\lambda_i]-$ the cell of periodicity and
$\Lambda_{diag}=(\lambda_1,\lambda_2,\dots,\lambda_d)-$ist diagonal.
Given $d-$vectors $\Bk=(k_1,k_2,\dots,k_d)$ and $\Bm=(m_1,m_2,\dots,m_d),$ denote
\begin{equation}
\label{B4.3}
\Bk\odot\Bm=(k_1m_1,k_2m_2,\dots,k_dm_d).
\end{equation}
The next theorem gives a necessary condition for coupling between the waves $u_1$ and $u_2.$
\begin{Theorem}
\label{BThm:4.1}
For the waves $(\Bk,\omega_1)$ and $ (\Bm,\omega_2)$ to couple, it is necessary, that $\omega_1=\omega_2$ and
$\frac{(\Bk-\Bm)\odot\Lambda_{diag}}{2\pi}\in \mathbb Z^d.$
\end{Theorem}

\begin{proof}
The proof of the theorem is based on Lemma~\ref{Blem:3.5}. We have to show, that the coefficients $d_{jp}^{(l)}$ vanish for $l\neq p,$
if one of the conditions in the theorem is not satisfied.
It is clear, that the integrand of $d_{jp}^{(l)}$ has the form $e^{\pm i(\omega_1-\omega_2)\xi_0}W(\xi')$, thus the integral
over the over the volume of $|Q|$ will vanish by Lemma~\ref{Blem:3.5}, as the coefficient of the exponent $e^{\pm i(\omega_1-\omega_2)\xi_0}$
does not depend on $\xi_0.$ On the other hand, if $\omega_1=\omega_2,$ then $d_{jp}^{(l)}$ will have the form
$e^{\pm i(\Bk-\Bm)\xi'}W(\xi')$, where $W(\xi')$ is a periodic function in $\xi'$ with the cell-period of that of the medium. Therefore,
again, an application of Lemma~\ref{Blem:3.5} completes the proof.
\end{proof}

\begin{Theorem}
\label{BThm:4.2}
Any two different waves $(\Bk,\omega_1)$ and $ (\Bm,\omega_2)$ do not couple.
\end{Theorem}

\begin{proof}
By Theorem~\ref{BThm:4.1}, for the waves $(\Bk,\omega_1)$ and $ (\Bm,\omega_2)$ to couple one must have $\omega_1=\omega_2$ and
$\frac{(\Bk-\Bm)\odot\Lambda_{diag}}{2\pi}\in \mathbb Z^d.$ We can without loss of generality assume, making a change of variables if
necessary, that the cell of periodicity $\Lambda$ of the medium is an $n-$dimensional unit cube, i.e.,
$\lambda_i=1$ for $i=1,2,\dots,d.$ Thus we have $\frac{(\Bk-\Bm)\odot\Lambda_{diag}}{2\pi}=\frac{\Bk-\Bm}{2\pi},$ thus
the condition $\frac{(\Bk-\Bm)\odot\Lambda_{diag}}{2\pi}\in \mathbb Z^d$ yields $k_i-m_i=2\pi q_i$ for $i=1,2,\dots,d$ and some
$q_i\in\mathbb Z.$ The last set of equations and the fact, that the medium cell of periodicity is a unit cube imply that the wave $u_1$ is a
scalar multiple of $u_2,$ which completes the proof.
\end{proof}
\begin{Remark}
\label{Brem:4.3}
For the coefficients $d_{jp}^{(l)}$ one has
$$d_{jp}^{(l)}=0,\quad\text{if}\quad l\neq p,$$
due to the non-coupling of different waves.
\end{Remark}
Combining now the above results with the result in Section~\ref{Bsec:2} we arrive at the following:
\begin{Theorem}
\label{BThm:4.4}
The effective equation associated to (\ref{B2.1}) for the wave $(\Bk,\omega)$ is given by
\begin{equation}
\label{B4.4}
\sum_{j=0}^dd_{j}\frac{\partial f_0(\BX)}{\partial X_j}=0,
\end{equation}
where the coefficients entering this homogenized equation are given by
\begin{equation}
\label{B4.5}
d_{j}=\lim_{Q\to\infty}\sum_{i=0}^d\frac{1}{|Q|}\int_{Q}U_0^{\ast}\left(2C_{ij}(\BGx')
\frac{\partial U_0(\BGx)}{\partial \xi_i}+\frac{\partial C_{ij}(\BGx')}{\partial \xi_i}U_0(\BGx)\right),\quad j=0,1,\dots,d,
\end{equation}
and $U_0$ solves the Bloch equation (\ref{B2.9}).
\end{Theorem}

The next theorem gives a simplification of formula (\ref{B4.5}).
\begin{Theorem}
\label{BThm:4.5}
The formula (\ref{B4.5}) can be simplified to
\begin{equation}
\label{B4.6}
d_{j}=\sum_{i=0}^d\frac{1}{|\Lambda|}\int_{\Lambda}U_0^{\ast}\left(2C_{ij}(\BGx')
\frac{\partial U_0(\BGx)}{\partial \xi_i}+\frac{\partial C_{ij}(\BGx')}{\partial \xi_i}U_0(\BGx)\right),\quad j=0,1,\dots,d,
\end{equation}
and $U_0$ solves the Bloch equation (\ref{B2.9}).
\end{Theorem}
\begin{proof}
The proof is a direct consequence of Lemma~\ref{Blem:3.5}. Recalling the formula (\ref{B2.10}) for the function $U_0(\xi)$ and plugging
in the expression of $U_0(\xi)$ into the formula (\ref{B4.5}) and calculating the partial derivatives, all the exponents cancel out and
one is left with a function $f(\xi')$ integrated over a time-space supercell $Q.$ The integration in time is balanced by the denominator $|Q|$ and thus
we are left with the integral of $f(\xi')$ over a space supercell. Finally, an application of Lemma~\ref{Blem:3.5} with the value $\xi=\mathbf{0}$ completes the proof.
\end{proof}
%%%%%%%%%%%%%%%%%%%%%%%%%%%%%%%%%%%%%%%%%%%%%%%%%%%%%%%%%%%%%
\section{The case of vector valued waves}
%%%%%%%%%%%%%%%%%%%%%%%%%%%%%%%%%%%%%%%%%%%%%%%%%%%%%%%%%%%%%%%%%%%%%%
\label{Bsec:5}
\setcounter{equation}{0}

In this section we allow for vector potentials $\Bu$ having $n$ components, and we
consider a system of $n$ equations in $d$ dimensions:
\begin{equation}
\label{B5.1}
\nabla\cdot\left(\mathbf{a}\left(\frac{\mathbf{x}}{\epsilon}\right)\nabla \mathbf{u}\right)=\mathbf{b}\left(\frac{\mathbf{x}}{\epsilon}\right)
\frac{\partial^2 \mathbf{u}}{\partial t^2},\qquad \mathbf{x}=(x_1,x_2,\dots,x_d)\in\mathbb R^d,
\end{equation}
which reads in components as follows:
\begin{equation}
\label{B5.2}
\sum_{j=1}^d \frac{\partial}{\partial x_j}\left(\sum_{k=1}^n\sum_{l=1}^da_{ijkl}\left(\frac{\mathbf{x}}{\epsilon}\right)\frac{\partial u_k}{\partial x_l}\right)=\sum_{k=1}^nb_{ik}\left(\frac{\mathbf{x}}{\epsilon}\right)\frac{\partial^2 u_k}{\partial t^2}
,\quad i=1,2,\dots n\quad\text{and}\quad \mathbf{x}\in\mathbb R^d,
\end{equation}
where $\mathbf{a}$ is a fourth order tensor that has the usual symmetry $a_{ijkl}=a_{klij},$ $\mathbf{b}\in \mathbb{R}^{n\times n}$ is a symmetric matrix, and $\mathbf{u}\colon\mathbb R^d\to\mathbb R^n$ is a vector field. It is assumed that $\mathbf{a}$ and $\mathbf{b}$ are cell-periodic. These
equations appear most naturally in the context of elastodynamics, where $\Bu(\Bx)$ is identified as the displacement field, $\Ba(\Bx)$ as the elasticity
tensor, having the additional symmetries $a_{ijkl}=a_{jikl}=a_{ijlk}$, and $\Bb(\Bx)$ is the (possibly anisotropic) density. The three-dimensional
electromagnetic equations of Maxwell can also
be expressed in this form (\citeAY{Milton:2006:CEM})
with $\Bu(\Bx)$ representing the electric field, $\Bb(\Bx)$ the dielectric tensor,
and the components of $\Ba$ being related to
the magnetic permeability tensor $\BGm(\Bx)$ through the equations,
\beq a_{ijkl}=-e_{ijp}e_{klq}\{\BGm^{-1}\}_{pq}, \eeq{5.2a}
in which $e_{ijp}$ is the completely antisymmetric Levi-Civita tensor, taking values $+1$ or $-1$ according to whether $ijp$ is an even or odd
permutation of $123$ and being zero otherwise.

Like in Section~\ref{Bsec:2} we rewrite system (\ref{B5.2}) in the following form
\begin{align}
\label{B5.3}
&\begin{pmatrix}
\frac{\partial}{\partial t}\\
\nabla
\end{pmatrix}\cdot \mathbf{C}
\begin{pmatrix}
\frac{\partial}{\partial t}\\
\nabla
\end{pmatrix}\mathbf{u}(\mathbf{x},t)=0,\qquad
\text{where}\qquad
\mathbf{C}=(C_{ijkl}),\\ \nonumber
& 1\leq i,k\leq n, \quad 0\leq j,l\leq d.
\end{align}
and the tensor $\mathbf{C}$ derives from the tensor $\mathbf{a}$ and the matrix $\mathbf{b}$ as follows:
\begin{equation}
\label{B5.4}
\begin{cases}
C_{ijkl}=0 & if \quad jl=0, j+l\geq 1,\\
C_{i0k0}=-b_{ik} & if \quad 1\leq i,k\leq n,\\
C_{ijkl}=a_{ijkl} & if \quad 1\leq i,j,k,l.
 \end{cases}
\end{equation}
Remembering that $t=x_0,$ we arrive at
\begin{equation}
\label{B5.5}
\overline{\nabla}\cdot(\mathbf{C}\overline{\nabla}\mathbf{u}(\mathbf{x},t))=0.
\end{equation}
Replacing now $\overline{\nabla}$ with $\overline{\nabla}_X+\frac{1}{\epsilon}\overline{\nabla}_{\xi},$ where $\mathbf{X}=(X_0,X_1,\dots,X_d)$ is the slow variable and $\BGx=(\xi_0,\xi_1,\dots,\xi_d)$ is the fast variable, we arrive at the system of equations
\begin{align}
\label{B5.6}
\overline{\nabla}_{\xi}\cdot(\mathbf{C}(\xi)\overline{\nabla}_{\xi}\mathbf{u}(\BX,\BGx))+
\epsilon\overline{\nabla}_{\xi}\cdot(\mathbf{C}(\xi)\overline{\nabla}_{X}\mathbf{u}(\BX,\BGx))\\ \nonumber
+\epsilon\overline{\nabla}_{X}\cdot(\mathbf{C}(\xi)\overline{\nabla}_{\xi}\mathbf{u}(\BX,\BGx))
+\epsilon^2\overline{\nabla}_{X}\cdot(\mathbf{C}(\xi)\overline{\nabla}_{X}\mathbf{u}(\BX,\BGx))=0.
\end{align}
We adopt the same strategy as in Section~\ref{Bsec:2}, but with a
slight difference: as we already know, that there is no coupling
between two different waves, we seek the solution to (\ref{B5.6}) in
the form of one wave (rather than the two of (\ref{B2.5})) corresponding to the pair $(\Bm,\omega)$ on the dispersion diagram:
\begin{equation}
\label{B5.7}
\mathbf{u}(\BX,\BGx)=\mathbf{u}(\BX,\BGx)
\end{equation}
where the vector $e^{-i(\Bm\cdot\BGx'-\omega\xi_0)}\Bu(\BX,\BGx)$ is periodic in
$\BGx'=(\xi_1,\xi_2,\dots,\xi_d)$ and independent of $\xi_0.$ Next, we assume that the vector $\mathbf{u}$ has the expansion
\begin{equation}
\label{B5.8}
\mathbf{u}(\BX,\BGx)=\mathbf{u}_0(\BX,\BGx)+\epsilon \mathbf{u}_1(\BX,\BGx)
+\epsilon^2 \mathbf{u}_2(\BX,\BGx)+\dots.
\end{equation}
At the zeroth order we get the system
$$\overline{\nabla}_{\xi}\cdot(\mathbf{C}(\BGx)\overline{\nabla}_{\xi}\mathbf{u}_0(\BX,\BGx))=0.$$
This has the solution $\mathbf{u}_0(\BX,\BGx)=f_0(\mathbf{X})\mathbf{U}_0(\BGx)$,
where $f_0$ is a scalar and $\mathbf{U_0}\colon\mathbb{R}^{d+1}\to\mathbb{R}^n$ is a vector such that
$e^{-i(\Bm\cdot\BGx'-\omega\xi_0)}\mathbf{U_0}(\BGx)$ is periodic in
$\BGx'=(\xi_1,\xi_2,\dots,\xi_d)$ and independent of $\xi_0,$ and that the vector $\mathbf{U_0}$ solves the
system of Bloch equations:
\begin{equation}
\label{B5.9}
\overline{\nabla}_{\xi}\cdot(\mathbf{C}(\BGx)\overline{\nabla}_{\xi}\mathbf{U}_0(\BGx))=0.
\end{equation}

At the first order we get the following system
\begin{equation}
\label{B5.10}
\overline{\nabla}_{\xi}\cdot(\mathbf{C}(\BGx)\overline{\nabla}_{\xi}\mathbf{u}_1(\BX,\BGx))=
-\left[\overline{\nabla}_{\xi}\cdot(\mathbf{C}(\BGx)\overline{\nabla}_{X}\mathbf{u}_0(\BX,\BGx))
+\overline{\nabla}_{X}\cdot(\mathbf{C}(\BGx)\overline{\nabla}_{\xi}\mathbf{u}_0(\BX,\BGx))\right].
\end{equation}
We can then calculate
$$\overline{\nabla}_{X}(f_0(\mathbf{X})\mathbf{U}_0(\BGx))=\left(\frac{\partial f_0(\mathbf{X})}{\partial X_l}U_0^{k}(\BGx)\right)_{lk},$$
thus
\begin{align}
\label{B5.11}
\overline{\nabla}_{X}\cdot(\mathbf{C}(\BGx)&\overline{\nabla}_{\xi}\mathbf{u}_0(\mathbf{X,\BGx}))
=\sum_{j=0}^d\frac{\partial}{\partial \xi_j}\left(\sum_{l=0}^d\sum_{k=1}^nC_{ijkl}(\BGx)\frac{\partial f_0(\mathbf{X})}{\partial X_l}U_0^{k}(\BGx)\right)\\ \nonumber
&=\sum_{j,l=0}^d\frac{\partial f_0(\mathbf{X})}{\partial X_l}\sum_{k=1}^n\left(\frac{\partial C_{ijkl}(\BGx)}{\partial \xi_j}U_0^{k}(\BGx)+C_{ijkl}(\BGx)\frac{\partial U_0^{k}(\BGx)}{\partial \xi_j}\right).
\end{align}
We have similarly
\begin{equation}
\label{B5.12}
\overline{\nabla}_{X}\cdot(\mathbf{C}(\BGx)\overline{\nabla}_{\xi}\mathbf{u}_0^(\BX,\BGx))
=\sum_{j,l=0}^d\frac{\partial f_0(\mathbf{X})}{\partial X_l}\sum_{k=1}^nC_{ilkj}(\BGx)\frac{\partial U_0^{k}(\BGx)}{\partial \xi_j},
\end{equation}
thus we finally obtain
\begin{align}
\label{B5.13}
&\overline{\nabla}_{\xi}\cdot(\mathbf{C}(\BGx)\overline{\nabla}_{\xi}\mathbf{u}_1^(\BX,\BGx))\\ \nonumber
&=\sum_{j,l=0}^d\frac{\partial f_0(\mathbf{X})}{\partial X_l}\sum_{k=1}^n
\left(\frac{\partial C_{ijkl}(\BGx)}{\partial_{\xi_j}}U_0^{k}(\BGx)+(C_{ijkl}(\BGx)+C_{ilkj}(\BGx))\frac{\partial U_0^{k}(\BGx)}{\partial \xi_j}\right),
\\ \nonumber
&i=1,2,\dots,n.
\end{align}
Proceeding like in Section~\ref{Bsec:2} we multiply the system (\ref{B5.13}) by the field $\mathbf{U_0^{\ast}}$
and then integrate the obtained identity over the cell $Q$ to eliminate the vector $\mathbf{u_1}$. This gives the
following result:
\begin{Theorem}
\label{BThm:5.1}
By the analogy of Theorem~\ref{BThm:4.5}, the system (\ref{B5.1}) homogenizes to the following equation:
\begin{equation}
\label{B5.14}
\sum_{l=0}^dd_{l}\frac{\partial f_0(\mathbf{X})}{\partial X_l}=0,
\end{equation}
where the coefficients $d_{l}$  entering the homgenized equation are given by the formulae:
\begin{align}
\label{B5.15}
&d_{l}=\frac{1}{|\Lambda|}\int_{\Lambda}\sum_{j=0}^d\sum_{i,k=1}^n
[\frac{\partial C_{ijkl}(\BGx)}{\partial_{\xi_j}}U_0^{k}(\BGx)U_0^{i\ast}(\BGx)+\\ \nonumber
&(C_{ijkl}(\BGx)+C_{ilkj}(\BGx))\frac{\partial U_0^{k}(\BGx)}{\partial \xi_j}U_0^{i\ast}(\BGx)]d\BGx,\\ \nonumber
&\quad l=1,2,\dots,d.
\end{align}
\end{Theorem}
%%%%%%%%%%%%%%%%%%%%%%%%%%%%%%%%%%%%%%%%%%%%%%%%%%%%%%%%%%%%%%%%%%%%%%%%%%%%%%%%%%%%%%%%

\section{A general case applicable to the Schr{\"o}dinger equation}
\label{Bsec:6}
\setcounter{equation}{0}
%%%%%%%%%%%%%%%%%%%%%%%%%%%%%%%%%%%%%%%%%%%%%%%%%%%%%%%%%%%%%%%%%%%%%%%%%%%%%%%%%%%%%%%%%%%
Let $\Bx=(x_0,x_1,\dots,x_d)$ and $\Bx'=(x_1,x_2,\dots,x_d).$ Here $x_0$ could represent the time, and $\Bx'$ the remaining spatial coordinates.
We aim to homogenize the problem
\begin{equation}
\label{B6.1}
\begin{pmatrix}
\BG\\
\nabla \cdot \BG\\
\end{pmatrix}=
\BL_\epsilon\left(\frac{\Bx'}{\epsilon}\right)
\begin{pmatrix}
\nabla u\\
u\\
\end{pmatrix},
\end{equation}
where $u(\Bx)\colon\mathbb R^{d+1}\to \mathbb R$ is the unknown,
$\BL_\epsilon(\Bx')\colon\mathbb R^d\to\mathbb R^{(d+2)(d+2)}$ is a Hermitian matrix that is cell-periodic in $\Bx'\in\mathbb R^{d}$.

To see the connection with the Schr{\"o}dinger equation, we let $\psi(\Bx)$ denote the wavefunction,
where $\Bx=(x_0,\Bx')$ and $x_0=t$ denotes the time coordinate while $\Bx'$ denotes the spatial coordinate, $V(\Bx')$ denote
the time independent electrical potential, $\BGF(\Bx')=(\GF_1(\Bx'),\GF_2(\Bx'),\GF_3(\Bx'))$ denote the time independent magnetic potential,
with $\Bb=\Curl\BGF$ the magnetic induction, $e$ denote the charge on the electron, and $m$ denote its mass. Using the Lorentz gauge, and
noting that $V(\Bx')$ is independent of time, $\BGF(\Bx')$ can be taken to have zero divergence.
Lets also choose units so that $\hbar$,
which is Planck's constant divided by $2\pi$, has the value 1. We assume both $V(\Bx')$ and $\BGF(\Bx')$ are periodic functions of $\Bx'$
with the same unit cell. Following \citeAPY{Milton:2016:ETC},
the Schr{\"o}dinger equation in a magnetic field can be written in the form
\beq
\begin{pmatrix} q_t\\
\Bq_x\\
\frac{\partial q_t}{\partial t}+\nabla'\cdot\Bq_x
\end{pmatrix}
=
\begin{pmatrix}
0 & 0 & -\frac{i}{2} \\
0 & \frac{-\BI}{2m} & \frac{i e \BGF}{2m} \\
+\frac{i}{2}  &\frac{-i e \BGF}{2m} & -e V
\end{pmatrix}
\begin{pmatrix}\frac{\partial\Gy}{\partial t}\\
\nabla'\Gy\\
\Gy
\end{pmatrix},
\eeq{A0.63}
where $q_t(t,\Bx)$ is a scalar field and $\Bq_x(t, \Bx')$ is a vector field, and where $\nabla'$, and $\nabla'\cdot$
are the gradient and divergence with respect to $\Bx'$.
Expanding out this in matrix form gives
\beqa
q_t &=& -\frac{i}{2}\Gy, \nonum
\Bq_x &=& -\frac{1}{2m}\nabla'\Gy +\frac{i e \BGF}{2m}\Gy,\nonum
\frac{\partial q_t}{\partial t}+\nabla'\cdot \Bq_x &=&
-\frac{(\nabla')^2\Gy}{2m}
+\frac{i\nabla'\cdot(e\BGF\Gy)}{2m}-\frac{i}{2}\frac{\partial\Gy}{\partial
t}= -\frac{i e \BGF}{2m}\nabla'\Gy+\frac{i}{2}\frac{\partial \Gy}{\partial
t}-e V \Gy,
\eeqa{A0.64}
where $(\nabla')^2$ is the Laplacian with respect to $\Bx'$.
Upon eliminating $\Bq_x$ and $q_t$ these imply the familiar form for Schr\"odinger's equation in a magnetic field:
\beq
i\frac{\partial \Gy}{\partial t} =
\frac{1}{2m}[i\nabla'+e\BGF]^2\Gy+e V \Gy.
\eeq{A0.65}
Setting $\BG=(q_t, \Bq_x)$, $\nabla=(\Md/\Md t,\nabla')$, and $u=\psi$ we see that Schr\"odinger's equation in a magnetic field
can be expressed in the form
\begin{equation}
\begin{pmatrix}
\BG\\
\nabla \cdot \BG\\
\end{pmatrix}=
\BL(\Bx')
\begin{pmatrix}
\nabla u\\
u\\
\end{pmatrix},\quad {\rm with}\quad ~\BL(\Bx')=\begin{pmatrix}
\Ba & \Bb(\Bx') \\
(\Bb(\Bx'))^{T\ast} & c(\Bx') \\
\end{pmatrix},
\end{equation}
where
\beq \Ba=\begin{pmatrix}
0 & 0 \\
0 &  \frac{-\BI}{2m}
\end{pmatrix},\quad \Bb(\Bx')=\begin{pmatrix}-\frac{i}{2} \\ \frac{i e \BGF(\Bx')}{2m}\end{pmatrix},\quad c(\Bx')=-e V(\Bx').
\end{equation}
With appropriate scaling, this is of the form (\ref{B6.1}).

The equation (\ref{B6.1}) will be called a constitutive relation, as it relates $u$ and its gradient $\nabla u$, to $\BG$ and
its divergence $\Div\BG$ through the matrix $\BL_\epsilon$. Let $\BX=(X_0,X_1,\dots,X_d)$ be the slow variable
and let $\BGx=\frac{\BX}{\epsilon}$ be the fast variable. Denote furthermore $\BGx'=(\Gx_1,\Gx_2,\dots,\Gx_d)$. We assume that the matrix $\BL_\epsilon$ has the form
\begin{equation*}
\BL_\epsilon=\begin{pmatrix}
\Ba(\BGx') & \Bb(\BGx')/\Ge\\
(\Bb(\BGx'))^{T\ast}/\Ge & c(\BGx')/\Ge^2\\
\end{pmatrix},
\end{equation*}
where we assume that $\Ba\in\mathbb R^{(d+1)\times (d+1)}$ is a real symmetric matrix, $\Bb\in\mathbb C^{(d+1)\times1}$ is a complex divergence
free field and $c\in\mathbb R$ is a real function. With our choice of the Lorentz gauge, $\Bb(\Bx)$ is divergence free for the
Schr{\"o}dinger equation in a magnetic field.

Next we expand $u$ and $\BG$ in powers of $\epsilon:$
\begin{align}
\label{B6.2}
\BG&=\BG_0+\epsilon \BG_1+\epsilon^2\BG_2+\dots\\ \nonumber
u&=u_0+\epsilon u_1+\epsilon^2u_2+\dots.
\end{align}
After replacing $\nabla$ by $\nabla_X+\frac{1}{\epsilon}\nabla_\xi$ and equating the coefficients
of the same power of $\epsilon$ on both sides of (\ref{B6.1}) we obtain the following equations in orders of $\Ge^{-1}$ and $\Ge^0$ respectively:\\
\begin{itemize}
\item{}[Order $\Ge^{-1}$].
\begin{align*}
\BG_0(\BX,\BGx)&=\Ba(\BGx')\nabla_\xi u_0(\BX,\BGx)+\Bb(\BGx')u_0(\BX,\BGx)\\
\nabla_\xi\cdot \BG_0&=c(\BGx')u_0(\BX,\BGx)+(\Bb(\BGx'))^\ast\nabla_\xi u_0(\BX,\BGx)
\end{align*}
from which we get the Bloch equation for $u_0:$
\begin{equation}
\label{B6.3}
\nabla_\xi\cdot(\Ba(\BGx')\nabla_\xi u_0)+(\Bb(\BGx')-(\Bb(\BGx'))^\ast)\cdot\nabla_\xi u_0-c(\BGx')u_0=0.
\end{equation}

\item{}[Order $\Ge^0$]. In the zeroth order we get the following system
\begin{align*}
\BG_1(\BX,\BGx)&=\Ba(\BGx')(\nabla_X u_0(\BX,\BGx)+\nabla_\xi u_1(\BX,\BGx))+\Bb(\BGx')u_1(\BX,\BGx)\\
\nabla_x\cdot \BG_0+\nabla_\xi\cdot \BG_1&=(\Bb(\BGx'))^\ast(\nabla_x u_0(\BX,\BGx)+\nabla_\xi u_1(\BX,\BGx))+c(\BGx')u_1(\BX,\BGx),
\end{align*}
from where we get by eliminating $\BG_0$ and $\BG_1,$
\begin{align}
\label{B6.4}
\nabla_\xi\cdot&(\Ba(\BGx')\nabla_\xi u_1)+(\Bb(\BGx')-(\Bb(\BGx'))^\ast)\cdot\nabla_\xi u_1-c(\BGx')u_1\\ \nonumber
&=-\nabla_X\cdot(\Ba(\BGx')\nabla_\xi u_0)-\nabla_\xi\cdot(\Ba(\BGx')\nabla_x u_0)-\nabla_X\cdot (\Bb(\BGx')u_0)+(\Bb(\BGx'))^\ast\cdot\nabla_X u_0.
\end{align}
Next we assume, that $u_j(\BX,\BGx)$ is such, that the functions
$e^{i(k\BGx'-\omega\xi_0)}u_j(\BX,\BGx)$ are periodic in
$\BGx'$ and do not depend on $\xi_0.$ We assume furthermore, that $u_0(\BX,\BGx)$ solves the Bloch
equation (\ref{B6.3}) and thus is separable in the fast and slow variables, namely we get
\begin{equation}
\label{B6.5}
u_0(\BX,\BGx)=U_0(\BGx')f_0(\BX).
\end{equation}
We have that
\begin{align*}
&-\nabla_x\cdot(\Ba(\BGx')\nabla_\xi u_0)-\nabla_\xi\cdot(\Ba(\BGx')\nabla_x u_0)-\nabla_X\cdot (\Bb(\BGx')u_0)+(\Bb(\BGx'))^\ast\cdot\nabla_X u_0=
\\ \nonumber
&=-\sum_{i,j=0}^d\frac{\partial f_0(\BX)}{\partial X_j}\left(2C_{ij}(\BGx')\frac{\partial U_0(\BGx)}{\partial \xi_i}
+\frac{\partial C_{ij}(\BGx')}{\partial \xi_i}U_0(\BGx)-(b_j(\xi')+(b_j(\xi'))^\ast)U_0\right),
\end{align*}
Thus we get combining with (\ref{B6.5}),
\begin{align}
\label{B6.6}
&\nabla_\xi\cdot(\Ba(\BGx')\nabla_\xi u_1)+(\Bb(\BGx')-(\Bb(\BGx'))^\ast)\cdot\nabla_\xi u_1-c(\BGx')u_1\\ \nonumber
&=-\sum_{i,j=0}^d\frac{\partial f_0(\BX)}{\partial X_j}\left(2C_{ij}(\BGx')\frac{\partial U_0(\BGx)}{\partial \xi_i}
+\frac{\partial C_{ij}(\BGx')}{\partial \xi_i}U_0(\BGx)-(b_j(\xi')+(b_j(\xi'))^\ast)U_0\right),
\end{align}
Next we multiply the equation (\ref{B6.6}) by $ U_0^\ast$ and integrate over $\overline{Q}$ to
eliminate $u_1$ and obtain the effective equation. We proceed by the analogy of (\ref{B2.14})-(\ref{B2.20}).
First, by taking the complex conjugate of the Bloch equation (\ref{B6.3}) we get
\begin{equation}
\label{B6.7}
\nabla_\xi\cdot(\Ba(\BGx')\nabla_\xi U_0^{\ast})-(\Bb(\BGx')-(\Bb(\BGx'))^\ast)\cdot\nabla_\xi U_0^{\ast}-c(\BGx')U_0^{\ast}=0,
\end{equation}
thus by multiplying equation (\ref{B6.7}) by $u_1$ and integrating over a rectangle $Q$
by parts and using the divergence-free property of $\Bb,$ we get
\begin{align}
\label{B6.8}
&0=\\ \nonumber
&\int_Qu_1(\nabla_\xi\cdot (\Ba(\BGx')\nabla_\xi U_0^{\ast})-(\Bb(\BGx')-(\Bb(\BGx'))^\ast)\cdot\nabla_\xi U_0^{\ast}-c(\BGx'))U_0^{\ast})d\BGx=\\ \nonumber
&\int_QU_0^{\ast}\left(\nabla_\xi\cdot(\Ba(\BGx')\nabla_\xi u_1)+U_0^{\ast}(\Bb(\BGx')-(\Bb(\BGx'))^\ast)\cdot\nabla_\xi u_1-c(\BGx')u_1\right)d\BGx
\\ \nonumber
&+\mathrm{surface\ \  term,}
\end{align}
thus by the analogy of (\ref{B2.14})-(\ref{B2.20}) we get
\begin{equation}
\label{B6.9}
\lim_{Q\to\infty}\frac{1}{|Q|}\int_QU_0^{\ast}\left(\nabla_\xi\cdot(\Ba(\BGx')\nabla_\xi u_1)+U_0^{\ast}(\Bb(\BGx')-(\Bb(\BGx'))^\ast)\cdot\nabla_\xi u_1-c(\BGx')u_1\right)d\BGx=0.
\end{equation}
Finally, combining (\ref{B6.9}) and (\ref{B6.6}) we arrive at the effective equation
\begin{equation}
\label{B6.10}
\sum_{j=0}^dd_{j}\frac{\partial f_0(\BX)}{\partial X_j}=0,
\end{equation}
where by the analogy of Theorem~\ref{BThm:4.5}, one has
\begin{equation}
\label{B6.11}
d_{j}=\sum_{i=0}^d\frac{1}{|\Lambda|}\int_{\Lambda}U_0^{\ast}\left(2C_{ij}(\BGx')
\frac{\partial U_0(\BGx)}{\partial \xi_i}+\frac{\partial C_{ij}(\BGx')}{\partial \xi_i}U_0(\BGx)-(b_j(\xi')+(b_j(\xi'))^\ast)U_0\right)d\BGx.
\end{equation}
\end{itemize}

%%%%%%%%%%%%%%%%%%%%%%%%%%%%%%%%%%%%%%%%%%%%%%%%%%%%%%%%%%%%%%%%%%%%%%%%%%%%%%%%%5
\section{Simplifying the effective equation}
\label{Bsec:7}
\setcounter{equation}{0}
%%%%%%%%%%%%%%%%%%%%%%%%%%%%%%%%%%%%%%%%%%%%%%%%%%%%%%%%%%%%%%%%%%%%%%%%%%%%%%%%%%%%55555

In this section we relate the dispersion relation $\omega=g(\Bk)$ and the effective coefficients.\\
\textbf{[The scalar case].}
Assume we have the effective equation (\ref{B4.4}) for a single wave $(\Bk,\omega).$ Identifying $X_0,X_1,\ldots X_d$
with $t,x_1,\ldots x_d$ we can rewrite it in the following way:
\begin{equation}
\label{B7.1}
\Ba_1\cdot\nabla f_0(t,\Bx)+b_1\frac{\partial f_0(t,\Bx)}{\partial t}=0.
\end{equation}
Assume $\epsilon>0$ is small enough, and suppose the pair $(\Bk+\epsilon\Gd \Bk,\Go+\Ge\Gd\omega)$
also lies on the dispersion relation. Since
$g(\Bk+\epsilon\Gd \Bk)= g(\Bk)+\epsilon\Gd \Bk\cdot\nabla g(\Bk)+\mathcal{O}(\epsilon^2),$
we have $\Gd\omega= \Gd \Bk\cdot\nabla g(\Bk)+\mathcal{O}(\Ge).$ We know one
solution of the wave equation is the Bloch solution
\beq u(\Bx,t)=e^{i[(\Bk+\Ge\Gd \Bk)\cdot(\Bx/\Ge)-(\omega\Ge\Gd\Go)(t/\Ge)]}V_\Ge(\Bx/\Ge),
\eeq{B7.0}
where with $\Bx/\Ge=\BGx'$,  $V_\Ge(\BGx')$ satisfies the Bloch equations
\beq (\omega+\Ge\Gd\Go)^2b(\BGx')V_\Ge(\BGx')
+(-i(\Bk+\Ge\Gd\Bk)+\overline{\nabla}_{\xi'})\cdot \mathbf{a}(\BGx')(-i(\Bk+\Ge\Gd\Bk)+\overline{\nabla}_{\xi'})V_\Ge(\BGx')=0
\eeq{B7.0a}
and $V_\Ge(\BGx)$ is periodic in $\BGx$.
With appropriate normalizations to ensure this has a unique solution for $V_\Ge(\BGx')$, we can write
\beq V_\Ge(\BGx')=V_0(\BGx')+\frac{\Md V_\Ge(\BGx')}{\Md \Ge}\Bigg|_{\Ge=0}\Ge+\mathcal{O}(\Ge^2).
\eeq{B7.0b}
So \eq{B7.0} has the expansion
\beq u(\Bx,t)=f(t,\Bx)U_0(\BGx')+\mathcal{O}(\Ge),\quad {\rm with}~~f(t,\Bx)=e^{i(\Gd \Bk\cdot\Bx-\Gd\omega t)}.
\eeq{B7.0c}

Then it is clear, that the function $f_0=e^{i(\Gd \Bk\cdot\Bx-\Gd\omega t)}$ must solve the equation (\ref{B7.1}),
from which we get
\begin{align*}
i(\Ba_1\cdot \Gd \Bk-b_1\nabla g(\Bk)\cdot \Gd \Bk)&=0,\quad \text{for all}\quad \Gd \Bk\in\mathbb R^d,\\ \nonumber
\end{align*}
from where we get
\begin{equation}
\label{B7.2}
\Ba_1=b_1\nabla g(\Bk).
\end{equation}
Thus the effective equation becomes
\beqa
&\nabla g\cdot\nabla f_0(t,\Bx)+\frac{\partial f_0(t,\Bx)}{\partial t}
=0,\\ \nonumber
\eeqa{B7.3}
Note that the solution of this equation is the travelling wave packet
$$f_0(t,\Bx)=h(\Bv\cdot\Bx-t)),$$
where $h$ is an arbitrary function that has first partial derivatives, and $\Bv$ is the group velocity which satisfies
$\Bv\cdot\nabla g=1$. As mentioned in the introduction this effective equation fails to capture dispersion
which is captured in the approach of \citeAPY{Allaire:2011:DGO}.

\textbf{[The vector case].} As the effective equations (\ref{B5.14}) in the vector case are exactly the same as in the scalar case, then we get the
 same relation as in the scalar case.

\section*{Acknowledgements}
G.W. Milton and D. Harutyunyan
are grateful to the University of Utah and to the National Science Foundation for support through grant DMS-1211359. G.W. Milton is grateful to Kirill Cherednichenko for explaining to him the results of \citeAPY{Birman:2006:HMP}, as summarized in the appendix. R.V. Craster thanks the EPSRC (UK) for their support through the Programme Grant EP/L024926/1.
We are grateful to Alexander Movchan and Stewart Haslinger for useful conversations about multipole techniques.

\section*{Author contributions}
All of the authors have provided substantial contributions to the conception and
design of the model, interpretation of the results, and writing the article.
All authors have given their final approval of the version to be published.

\section*{Conflict of interests} The authors of the paper have no competing interests.

\section*{Data accessibility} This paper has no data.

\section*{Funding statement} G.W. Milton and D. Harutyunyan are grateful to the National Science Foundation for support through grant DMS-1211359. R.V. Craster thanks the EPSRC (UK) for their support through the Programme Grant EP/L024926/1.

\section*{Ethics statement} This paper does not involve any collection of human data.

%%%%%%%%%%%%%%%%%%%%%%%%%%%%%%%%%%%%%%%%%%%%%%%%%%%%%%%%%%%%%%%%%%%%%%%%%%%%%%%%%%%%%%%
\section{Appendix A}
%%%%%%%%%%%%%%%%%%%%%%%%%%%%%%%%%%%%%%%%%%%%%%%%%%%%%%%%%%%%%%%%%%%%%%%%%%%%%%%%%%%%%%

Here we make the connection between the results of \citeAPY{Birman:2006:HMP} and those of \citeAPY{Craster:2010:HFH}. The first thing that is relevant is equation (1.12) of
\citeAPY{Birman:2006:HMP}, where they expand at the edge $E_s$ of a band-gap (where $E_s$ may represent an energy, or frequency) a minimum or maximum of the dispersion diagram as a quadratic form, involving quadratic functions $b^{(\pm)}$. These quadratic
functions determine the "effective coefficients" that enter the homogenized equations of  \citeAPY{Craster:2010:HFH}. In that formula (1.12) the $\xi^{(\pm)}$ is the wave vector $\Bk=\xi$, one
expands around. (They assume there may be $j=1,2,\ldots m_{\pm}$ such wavevectors attaning the same energy  $E_s$, but here, for simplicity, we assume there is just one.) The $\psi_{s\pm}(\Bx,\BGx)$ at the top of page 3685 is the eigenfunction, or Bloch function, associated with $E_s$ . The main result is that the resolvent (2.1) approaches (2.2).The connection is clearer if one writes out what this means. Let us suppose there is a source term $g(\Bx)$. Then
if you are interested in solving $[A-(\Gl-\Ge^2\kappa^2)]u=g$, where $\kappa$ is chosen so $(\Gl-\kappa^2)$ is in the gap, and $\Ge\in (0,1]$,
the solution is $u=S(\Ge)g$, where $S(\Ge)$ is the resolvent.  Birman and Suslina say that when $\Ge$ is small, the result is approximately the same as solving $u=S^0(\Ge) g$, i.e.
\beq [b_j(D)+\Ge^2](u/\psi_{s\pm})=(g/\psi_{s\pm})
\eeq{B.8.0}
Here $u/\psi_{s\pm}$ can be identified with the modulating function $f$ of  \citeAPY{Craster:2010:HFH}, $b_j(D)$ is the effective operator, $D$ is the operator $-i\Grad$ (see
point 3 in the introduction). Thus the analysis of  \citeAPY{Birman:2006:HMP} applies even when there are source terms $g\ne 0$ and allows for expansion points $\xi^{(\pm)}$
which are  not necessarily at $\Bk=\xi^{(\pm)}=0$ or at the edge of the Brillouin zone. The reason  \citeAPY{Birman:2006:HMP} assume one is in the gap is to make sure the solution is localized, which is easier for the mathematical analysis.

\section{Appendix B}
\label{Bsec:3}

\begin{Definition}
\label{Bdef:3.1}
Assume $Q=\prod_{i=1}^d[a_i,b_i]\subset\mathbb R^d$ is a rectangle. Then we write $Q\to\infty$ if
$b_i-a_i\to\infty$ for all $i\in\{1,2,\dots,d\}.$
\end{Definition}

The next two lemmas will be crucial in the process of homogenization.
\begin{Lemma}
\label{Blem:3.2}
Assume $f\colon\mathbb R\to \mathbb R$ is periodic with a period $T>0$ and $f\in L^2(0,T).$
Then for any $b\neq 0$ there holds:
\begin{align}
\label{B3.1}
&\lim_{a\to\infty}\frac{1}{a}\int_{0}^af(x)e^{ibx}dx=0,\quad \text{if} \quad \frac{Tb}{2\pi}\notin\mathbb Z,\\ \nonumber
&\lim_{a\to\infty}\frac{1}{a}\int_{0}^af(x)e^{ibx}dx=\frac{1}{T}\int_{0}^Tf(x)e^{ibx}dx,\quad \text{if} \quad \frac{Tb}{2\pi}\in\mathbb Z.
\end{align}
\end{Lemma}

\begin{proof}

Note, that if $a=mT+r,$ where $0\leq r<T$ and $m\in\mathbb Z, m\geq 0,$ then we have
\begin{align}
\label{B3.2}
\frac{1}{a}\int_{0}^af(x)e^{ibx}dx&=\frac{1}{a}\int_{0}^{mT+r}f(x)e^{ibx}dx\\ \nonumber
&=\frac{1}{a}\int_{mT}^{mT+r}f(x)e^{ibx}dx+\frac{1}{mT+r}\sum_{j=0}^{m-1}\int_{jT}^{(j+1)T}f(x)e^{ibx}dx\\ \nonumber
&=\frac{1}{a}\int_{mT}^{mT+r}f(x)e^{ibx}dx+\frac{1}{mT+r}\sum_{j=0}^{m-1}\int_{0}^{T}f(x)e^{ib(x+jT)}dx\\ \nonumber
&=\frac{1}{a}\int_{mT}^{mT+r}f(x)e^{ibx}dx+\frac{1}{mT+r}\sum_{j=0}^{m-1}e^{ibTj}\int_{0}^{T}f(x)e^{ibx}dx.
\end{align}
We have by the Schwartz inequality, that
\begin{equation}
\label{B3.3}
\left|\frac{1}{a}\int_{mT}^{mT+r}f(x)e^{ibx}dx\right|\leq \frac{1}{a}\int_{0}^{T}|f(x)|dx\leq \frac{\sqrt{T}}{a}\|f\|_{L^2(0,T)}\to 0,\quad\text{as}\quad a\to\infty.
\end{equation}
On the other hand we have
\begin{align}
\label{B3.4}
&\frac{1}{mT+r}\sum_{j=0}^{m-1}e^{ibTj}\int_{0}^{T}f(x)e^{ibx}dx=\frac{m}{mT+r}\int_{0}^{T}f(x)e^{ibx}dx,\quad\text{if}\quad bT=2\pi l,\\ \nonumber
&\frac{1}{mT+r}\sum_{j=0}^{m-1}e^{ibTj}\int_{0}^{T}f(x)e^{ibx}dx=\frac{(1-e^{ibTm})}{(mT+r)(1-e^{ibT})}\int_{0}^{T}f(x)e^{ibx}dx,
\quad\text{if}\quad bT\neq2\pi l.
\end{align}
In the first case we get
\begin{align}
\label{B3.5}
\lim_{a\to\infty}\frac{1}{a}\int_{0}^af(x)e^{ibx}dx&=\lim_{m\to\infty}\frac{m}{mT+r}\int_{0}^{T}f(x)e^{ibx}dx\\ \nonumber
&=\frac{1}{T}\int_{0}^Tf(x)e^{ibx}dx,
\end{align}
In the second case we have again by the Schwartz inequality, that
\begin{equation}
\label{B3.6}
\left|\frac{(1-e^{ibTm})}{(mT+r)(1-e^{ibT})}\int_{0}^{T}f(x)e^{ibx}dx\right|\leq \frac{2\sqrt{T}\|f\|_{L^2(0,T)}}{a|1-e^{ibT}|},
\end{equation}
thus we get
$$
\lim_{a\to\infty}\frac{1}{a}\int_{0}^af(x)e^{ibx}dx=0.
$$

\end{proof}

The next lemma is generalization of Lemma~\ref{Blem:3.2}.
\begin{Lemma}
\label{Blem:3.3}
Let the functions $f,g\colon\mathbb R\to \mathbb R$ have periods $T_1,T_2>0$ respectively. Assume
that $f\in L^2(0,T_1)$ and $g\in L^2(0,T_2)$ and
\begin{equation}
\label{B3.7}
\int_{0}^{T_1}f(x)dx=0.
\end{equation}
Then one has:

\begin{align}
\label{B3.8}
&\lim_{a\to\infty}\frac{1}{a}\int_{0}^af(x)g(x)dx=0,\quad \text{if} \quad \frac{T_1}{T_2}\notin\mathbb Q,\\ \nonumber
&\lim_{a\to\infty}\frac{1}{a}\int_{0}^af(x)g(x)dx=\frac{1}{nT_1}\int_{0}^{nT_1}f(x)g(x)dx,\quad \text{if} \quad \frac{T_1}{T_2}=\frac{m}{n},
\ \ m,n\in\mathbb Z.
\end{align}
\end{Lemma}

\begin{proof}
Assume first that $\frac{T_1}{T_2}=\frac{m}{n},$ where $m,n\in\mathbb N,$ thus $nT_1=mT_2.$ We have for any
$a>nT_1,$ that $a=knT_1+r,$ where $0\leq r<nT_1$ and $k\in\mathbb N.$ Then we have by the periodicity of $f$ and $g,$ that
\begin{align}
\label{B3.9}
\frac{1}{a}\int_{0}^af(x)g(x)dx&=\frac{1}{knT_1+r}\int_{0}^{knT_1}f(x)g(x)dx+\frac{1}{knT_1+r}\int_{knT_1}^{knT_1+r}f(x)g(x)dx\\ \nonumber
&=\frac{k}{knT_1+r}\int_{0}^{nT_1}f(x)g(x)dx+\frac{1}{knT_1+r}\int_{knT_1}^{knT_1+r}f(x)g(x)dx.
\end{align}
It is clear that
\begin{align*}
\lim_{a\to\infty}\frac{k}{knT_1+r}\int_{0}^{nT_1}f(x)g(x)dx&=\lim_{k\to\infty}\frac{k}{knT_1+r}\int_{0}^{nT_1}f(x)g(x)dx\\
&=\frac{1}{nT_1}\int_{0}^{nT_1}f(x)g(x)dx,
\end{align*}
and by the Schwartz inequality
\begin{align*}
\left|\frac{1}{knT_1+r}\int_{knT_1}^{knT_1+r}f(x)g(x)dx\right|&\leq \frac{1}{knT_1}\int_{0}^{nT_1}|f(x)g(x)|dx\\
&\leq \frac{1}{knT_1}\|f(x)\|_{L^2(0,nT_1)}\|g(x)\|_{L^2(0,nT_1)}\\
&\to 0
\end{align*}
as $k\to\infty,$ thus the case $\frac{T_1}{T_2}=\frac{m}{n}$ is proven.
Assume now that $\frac{T_1}{T_2}\notin\mathbb Q.$ By the Fourier expansion we have that
$$f(x)=\sum_{n=-\infty}^{\infty}a_ne^{\frac{2i\pi nx}{T_1}}$$
in the $L^2(0,T_1)$ sense. Denote $P_n(x)=\sum_{k=-n}^{n}a_ke^{\frac{2i\pi kx}{T_1}},$ then
$$P_n(x)\to f(x)\quad\text{in}\quad L^2(0,T_1),$$
thus for any $\epsilon>0$ there exists $N\in\mathbb N$ such that
\begin{equation}
\label{B3.10}
\|f(x)-P_N(x)\|_{L^2(0,T_1)}\leq \epsilon.
\end{equation}
If $a=k_1T_1+r_1=k_2T_2+r_2$ where $k_1,k_2\in\mathbb N$ and $0\leq r_1<T_1$, $0\leq r_2<T_2$, then
we have by the Schwartz inequality that for big enough $a$ there holds,
\begin{align}
\label{B3.11}
\frac{1}{a}&\left|\int_0^a f(x)g(x)dx-\int_0^a P_N(x)g(x)dx\right|\\ \nonumber
&\leq \frac{1}{a}\int_0^a |f(x)-P_N(x)||g(x)|dx\\ \nonumber
&\leq \frac{1}{a} \|f(x)-P_N(x)\|_{L^2(0,a)}\|g(x)\|_{L^2(0,a)}\\ \nonumber
&\leq \frac{1}{a} \sqrt{k_1+1}\|f(x)-P_N(x)\|_{L^2(0,T_1)}\sqrt{k_2+1}\|g(x)\|_{L^2(0,T_2)}\\ \nonumber
&\leq \frac{\epsilon \|g(x)\|_{L^2(0,T_2)}}{\sqrt{T_1T_2}}\frac{\sqrt{(k_1+1)(k_2+1)}}{\sqrt{k_1k_2}}\\ \nonumber
&\leq \frac{2\epsilon \|g(x)\|_{L^2(0,T_2)}}{\sqrt{T_1T_2}},
\end{align}
which implies, that it suffices to prove the lemma for $P_N(x)$ instead of $f(x).$ From the condition
$\int_0^{T_1} f(x)dx=0$ we get $a_0=0,$ thus
$$P_N(x)=\sum_{k=-N}^{-1}a_ke^{\frac{2i\pi kx}{T_1}}+\sum_{k=1}^{N}a_ke^{\frac{2i\pi kx}{T_1}}.$$
Now, an application of Lemma~\ref{Blem:3.2} to each of the summands $a_ke^{\frac{2i\pi kx}{T_1}}$ completes the proof.
\end{proof}

\begin{Lemma}
\label{Blem:3.4}
Let $f,g\colon\mathbb R\to\mathbb R$ and $T_1,T_2>0$ be such that $f(x)$ is $T_1-$periodic, $g(x)$ is $T_2-$periodic and
$\frac{T_1}{T_2}\notin\mathbb Q.$ Assume furthermore, that $f(x)\in W^{1,2}(0,T_1)$ and $g(x)\in L^2(0,T_2).$
Then
\begin{equation}
\label{B3.12}
\lim_{a\to\infty}\frac{1}{a}\int_{0}^a f'(x)g(x)dx=0.
\end{equation}
\end{Lemma}

\begin{proof}
The proof directly follows from Lemma~\ref{Blem:3.3} as $\int_{0}^{T_1}f'(x)dx=0$ by the periodicity of $f.$
\end{proof}

\begin{Lemma}
\label{Blem:3.5}
Assume the function $f(\BGx)\colon\mathbb R^d\to\mathbb R$ is cell-periodic and continuous with a cell of periodicity
$R=\prod_{i=1}^d [0,T_i].$ Then for any vector $\lambda=(\lambda_1,\lambda_2,\dots,\lambda_d)\in\mathbb R^d$ one has
\begin{align*}
&\lim_{Q\to\infty}\frac{1}{|Q|}\int_{Q}f(\BGx)e^{i\lambda\cdot\BGx}d\BGx=0,\quad \text{if} \quad T_j\lambda_j\neq 2\pi l,\ \ \in\mathbb Z
\quad\text{for some}\quad j\in\{1,2,\dots,d\},\\ \nonumber
&\lim_{Q\to\infty}\frac{1}{|Q|}\int_{Q}f(\BGx)e^{i\lambda\cdot\BGx}d\BGx=\frac{1}{|R|}\int_{R}f(\BGx)e^{i\lambda\cdot\BGx}d\BGx,
\quad \text{if} \quad T_j\lambda_j=2\pi l_j, l_j\in\mathbb Z,\ \ j=1,2,\dots,d.
\end{align*}
\end{Lemma}

\begin{proof}
The proof is straightforward as this is a consequence of the previous Lemma. It is easy to see, that
$$\lim_{Q\to\infty}\frac{1}{|Q|}\int_{Q}f(\BGx)e^{i\lambda\cdot\BGx}d\BGx=\lim_{l\to\infty}\frac{1}{l^d|R|}\int_{l\cdot R}f(\BGx)e^{i\lambda\cdot\BGx}d\BGx,$$
where $l\in\mathbb N.$ If $T_i\lambda_i=2\pi l_i, l_i\in\mathbb Z,\ i=1,2,\dots,d$ then we have
$$\frac{1}{l^d|R|}\int_{l\cdot R}f(\BGx)e^{i\lambda\cdot\BGx}d\BGx=\frac{1}{|R|}\int_{R}f(\BGx)e^{i\lambda\cdot\BGx}d\BGx,$$
for all $l\in\mathbb N.$ Assume now the set
 $I=\{j\ : \ T_j\lambda j\neq 2\pi l, l\in\mathbb Z\}\cap\{1,2,\dots,d\}$ is not empty. Then we have by the analogy of the proof of Lemma~\ref{Blem:3.2}
 and the Fubini theorem, that
\begin{equation}
\label{B3.13}
\frac{1}{l^d|R|}\int_{l\cdot R}f(\BGx)e^{i\lambda\cdot\BGx}d\BGx\leq \frac{C}{l^{|I|}}\to 0\quad\text{as}\quad l\to\infty,
\end{equation}
where $C$ is a constant depending on the value $M=\max_{R}|f(x)|.$ The proof is finished now.
\end{proof}

\bibliographystyle{chicago}
%\bibliographystyle{siam}
%\bibliographystyle{/u/ma/milton/tex/mod-xchicago}
%\bibliography{/u/ma/milton/tcbook,/u/ma/milton/newref}
%\bibliography{tcbook,newref}
%\bibliographystyle{/home/milton/latex/unimode/mod-xchicago}
\bibliography{newrefdavitnew}
\end{document}